\documentclass[a4paper,11pt]{amsart}
\usepackage{latexsym}
\usepackage{amscd}
\usepackage{amssymb}
\usepackage{amsfonts}
\usepackage{mathrsfs}
\usepackage[all]{xy}
\usepackage{amsmath, pb-diagram}
\usepackage{enumerate}

\DeclareMathAlphabet{\eusm}{OT1}{eusm}{m}{n}
\newtheorem{theor}{Theorem}[section]
\newtheorem{prop}[theor]{Proposition}
\newtheorem{defi}[theor]{Definition}
\newtheorem{cor}[theor]{Corollary}

\newtheorem{exam}[theor]{Example}
\newtheorem{remark}[theor]{Remark}
\def\Rad{\mbox{Rad\/}}

\newcommand{\Hom}{{\rm Hom}}
\newcommand{\Mod}{{\rm Mod}}

\newcommand{\fp}{{\rm fp}}
\newcommand{\op}{{\rm op}}
\newcommand{\Ab}{{\rm Ab}}

\begin{document}

\title[Relatively divisible and relatively flat objects]
{Relatively divisible and relatively flat objects \\ in exact categories: Applications}



\author[S. Crivei]{Septimiu Crivei}
\address {Faculty of Mathematics and Computer Science, Babe\c{s}-Bolyai University, 400084 Cluj-Napoca, Romania}
\email{crivei@math.ubbcluj.ro}

\author[D. Kesk\.{i}n T\"ut\"unc\"u]{Derya Kesk\.{i}n T\"ut\"unc\"u}
\address {Department of Mathematics, Hacettepe University, 06800 Ankara, Turkey} \email{keskin@hacettepe.edu.tr}

\begin{abstract} We continue our study of relatively divisible and relatively flat objects in exact categories in the
sense of Quillen with several applications to exact structures on finitely 
accessible additive categories and module categories. We derive consequences for exact structures generated by the
simple modules and the modules with zero Jacobson radical.
\end{abstract}

\date{October 27, 2018}

\maketitle

\section{Introduction}

In our paper \cite{CK1} we have introduced and studied relatively divisible and relatively flat objects in exact categories in the
sense of Quillen. For every relative cotorsion pair $(\mathcal{A},\mathcal{B})$ in an exact category $\mathcal{C}$, 
we have seen that $\mathcal{A}$ coincides with the class of relatively flat objects of $\mathcal{C}$ 
for some relative projectively generated exact structure, while $\mathcal{B}$ coincides with 
the class of relatively divisible objects of $\mathcal{C}$ for some relative injectively generated exact structure.
We have also established closure properties and characterizations in terms of approximation theory. 
For motivation and background the reader is referred to \cite{CK1}. 
In the present paper we give some representative applications of our theory to exact structures on finitely 
accessible additive categories and module categories. 

In Section 2 we show an application to absolutely pure objects in 
finitely accessible additive categories. As a consequence of a result for 
$\mathcal{D}$-$\mathcal{E}$-$\mathcal{A}$-divisible objects in functor categories, we
generalize a theorem known for finitely accessible abelian categories \cite[Proposition~5.6]{Prest}, and we deduce that
an object of a finitely accessible additive category is absolutely pure if and only if its pure-injective envelope is
injective.

Section 3 contains our main applications to module categories. We also consider flatly generated exact structures, and
we show and use their relationship with projectively generated ones. We establish several results connecting
purity-related notions with relatively divisible and relatively flat properties. We obtain some characterizations of
right IF rings, right FGF rings and right $\mathcal{N}$-coherent rings for some class $\mathcal{N}$ of finitely
presented right $R$-modules.

In the final two sections we further illustrate our theory with a series of properties which can be deduced from
our previous results. Motivated by some recent results by Fuchs \cite{Fuc}, Section 4 discusses the case of 
exact structures projectively, injectively and flatly generated by the class of simple modules, 
and we deduce and complete a series of known properties. Section 5 deals with exact structures projectively, 
injectively and flatly generated by the class of modules with zero Jacobson radical, and all results are new.

\section{Applications to finitely accessible additive categories}

We recall some terminology on finitely accessible additive categories, mainly from \cite{Prest}. An additive category
$\mathcal{C}$ is called \emph{finitely accessible} if it has direct limits, the class of finitely presented objects is
skeletally small, and every object is a direct limit of finitely presented objects. Recall that an object $C$ of
$\mathcal{C}$ is finitely presented if the representable functor $\Hom_{\mathcal{C}}(C,-):\mathcal{C}\to {\rm Ab}$
commutes with direct limits, where Ab is the category of abelian groups.

Let $\mathcal{C}$ be a finitely accessible additive category. By a \emph{sequence} 
$0\to X\stackrel{f}\to Y\stackrel{g}\to Z\to 0$ in $\mathcal{C}$ we mean a pair of composable morphisms $f:X\to Y$ and
$g:Y\to Z$ such that $gf=0$. The above sequence in $\mathcal{C}$ is called \emph{pure exact} if it induces an exact
sequence of abelian groups $$0\to \Hom_{\mathcal{C}}(P,X)\to \Hom_{\mathcal{C}}(P,Y)\to \Hom_{\mathcal{C}}(P,Z)\to 0$$
for every finitely presented object $P$ of $\mathcal{C}$. This implies that $f$ and $g$ form a kernel-cokernel pair,
that $f$ is a monomorphism and $g$ an epimorphism. In such a pure exact sequence $f$ is called a \emph{pure
monomorphism} and $g$ a \emph{pure epimorphism}. 

Any finitely accessible additive category $\mathcal{C}$ may be embedded as a full
subcategory of the functor category $(\fp(\mathcal{C})^{\op},\Ab)$ of all contravariant additive functors from the full
subcategory $\fp(\mathcal{C})$ of finitely presented objects of $\mathcal{C}$ to the category ${\rm Ab}$ of abelian
groups. In this way the pure exact sequences in $\mathcal{C}$ are those which become exact sequences in
$(\fp(\mathcal{C})^{\op},\Ab)$ through the embedding, and $\mathcal{C}$ may be seen as being equivalent to the full
subcategory of flat objects from $(\fp(\mathcal{C})^{\op},\Ab)$. The relevant functor inducing the equivalence is the
covariant Yoneda functor $H:\mathcal{C}\to (\fp(\mathcal{C})^{\op},\Ab)$, given on objects by $X\mapsto
H_X=\Hom_{\mathcal{C}}(-,X)\mid_{\fp(\mathcal{C})}$ and correspondingly on morphisms. Then $H$ preserves and reflects
purity. Moreover, the pure exact sequences define an exact structure in any finitely accessible additive category
$\mathcal{C}$, because $\mathcal{C}$ is equivalent to an extension closed full subcategory of an abelian category (e.g.,
see \cite[Lemma~10.20]{Buhler}). Note that every category of modules over a ring with identity, and more generally,
every functor category is finitely accessible.  

Recall some purity-related concepts in finitely accessible additive categories \cite{C14,Prest} and Grothendieck
categories \cite{St68,Xu}. 

\begin{defi} \rm Let $\mathcal{C}$ be a finitely accessible additive category. An object $X$ of $\mathcal{C}$ is called:
\begin{enumerate}
\item \emph{pure-injective} if it is injective with respect to every pure monomorphism in $\mathcal{C}$.
\item \emph{absolutely pure} if every monomorphism $X\to Y$ in $\mathcal{C}$ is pure. 
\end{enumerate}
\end{defi}

\begin{defi} \rm Let $\mathcal{G}$ be a Grothendieck category. An object $M$ of $\mathcal{G}$ is called:
\begin{enumerate}
\item \emph{flat} if every epimorphism $A\to M$ in $\mathcal{G}$ is pure.
\item \emph{cotorsion} if ${\rm Ext}^1_{\mathcal{G}}(A,M)=0$ for every flat object $A$ of $\mathcal{G}$.
\item \emph{weakly absolutely pure} if every monomorphism $M\to A$ in $\mathcal{G}$ with $A$ flat is pure.
\end{enumerate}
\end{defi}

\begin{remark} \rm (1) Absolutely pure objects of a Grothendieck category $\mathcal{G}$ coincide with its
$\mathcal{D}$-$\mathcal{E}$-divisible objects, where $\mathcal{D}$ is the exact structure given by all short exact
sequences and $\mathcal{E}$ is the exact structure given by all pure exact sequences.

(2) Weakly absolutely pure objects of a Grothendieck category $\mathcal{G}$ coincide with its
$\mathcal{D}$-$\mathcal{E}$-$\mathcal{A}$-divisible objects, where $\mathcal{D}$ is the exact structure given by all
short exact sequences, $\mathcal{E}$ is the exact structure given by all pure exact sequences and $\mathcal{A}$ is the
class of flat objects.
\end{remark}

Now we may deduce the following known characterization of absolutely pure objects in finitely accessible abelian
categories.

\begin{prop} \cite[Proposition~5.6]{Prest} Let $\mathcal{C}$ be a finitely accessible abelian category. The following
are equivalent for an object $X$ of $\mathcal{C}$:
\begin{enumerate}[(i)]
 \item $X$ is absolutely pure.
 \item The pure-injective envelope of $X$ is injective.
 \item $X$ is a pure subobject of an injective object. 
 \item $X$ is a pure subobject of an absolutely pure object. 
\end{enumerate}
\end{prop}

\begin{proof} Consider the exact structures on $\mathcal{C}$ given by the classes $\mathcal{D}$ of all short exact
sequences and $\mathcal{E}\subseteq \mathcal{D}$ of pure exact sequences, and the $\mathcal{D}$-perfect
$\mathcal{D}$-cotorsion pair $(\mathcal{A},\mathcal{B})$ in $\mathcal{C}$, where $\mathcal{A}$ is the
class of all objects and $\mathcal{B}$ is the class of pure-injective objects of $\mathcal{C}$. Then the class of 
$\mathcal{D}$-$\mathcal{E}$-$\mathcal{A}$-divisible objects of $\mathcal{C}$ coincides with the class of absolutely pure objects.
Now use \cite[Theorem~4.9]{CK1} (see also \cite[Proposition~4.6]{CK1}). 
\end{proof}

We are interested in showing a similar result in the case of arbitrary finitely accessible additive categories. In the
above proposition we have essentially used the fact that all short exact sequences yield an exact structure on any
abelian category. Such a result is no longer true if the category is not (at least) quasi-abelian (see \cite[Example~2.3]{CK1}). 
In order to be able to remove the hypothesis on $\mathcal{C}$ to be abelian, we need the following
proposition on weakly absolutely pure objects in functor categories, which can be immediately deduced from our
results. Note that it was proved directly in \cite{C14}.

\begin{prop} \cite[Theorem~3.1]{C14} \label{p:charmod} Let $\mathcal{S}$ be a small preadditive category. The following
are equivalent for an object $K$ of $(\mathcal{S}^{\op},\Ab)$:
\begin{enumerate}[(i)]
 \item $K$ is weakly absolutely pure.
 \item The cotorsion envelope of $K$ is weakly absolutely pure.
 \item There is a pure exact sequence $0\to K\to M\to N\to 0$ with $M$ weakly absolutely pure cotorsion and $N$ flat. 
 \item There is a pure exact sequence $0\to K\to M\to N\to 0$ with $M$ weakly absolutely pure and $N$ flat. 
\end{enumerate}
\end{prop}

\begin{proof} Consider the exact structures on $(\mathcal{S}^{\rm op},\Ab)$ given by the classes $\mathcal{D}$ of all
short exact sequences and $\mathcal{E}\subseteq \mathcal{D}$ of pure exact sequences, and the $\mathcal{D}$-perfect
$\mathcal{D}$-cotorsion pair $(\mathcal{F},\mathcal{C})$ in $(\mathcal{S}^{\rm op},\Ab)$ \cite[Theorem~3]{BEE} (also
see \cite[Corollary~3.3]{CPT}), where $\mathcal{F}$ is the class of flat objects and $\mathcal{C}$ is the class of
cotorsion objects. Note that $\mathcal{F}$ coincides with the class of $\mathcal{D}$-$\mathcal{E}$-flat objects in
$(\mathcal{S}^{\rm op},\Ab)$ and weakly absolutely pure objects are the same as
$\mathcal{D}$-$\mathcal{E}$-$\mathcal{A}$-divisible objects, where $\mathcal{A}=\mathcal{F}$. Then use \cite[Theorem~4.9]{CK1}. 
\end{proof}

We recall the following theorem, which relates purity-related properties of objects of a finitely accessible
additive category and of objects of the associated functor category. It shows why weakly absolutely pure objects appear
naturally in functor categories.

\begin{theor} \cite[Theorem~2.4]{C14} \label{t:equivclasses} Let $\mathcal{C}$ be a finitely accessible additive
category. The equivalence induced by the covariant Yoneda functor $H:\mathcal{C}\to (\fp(\mathcal{C})^{\op},\Ab)$
between $\mathcal{C}$ and the full subcategory of $(\fp(\mathcal{C})^{\op},\Ab)$ consisting of the flat objects
restricts to equivalences between the following full subcategories:
\begin{enumerate}[(i)]
\item pure-injective objects of $\mathcal{C}$ and cotorsion flat objects of $(\fp(\mathcal{C})^{\op},\Ab)$.
\item absolutely pure objects of $\mathcal{C}$ and weakly absolutely pure flat objects of
$(\fp(\mathcal{C})^{\op},\Ab)$.
\item injective objects of $\mathcal{C}$ and cotorsion weakly absolutely pure flat objects of
$(\fp(\mathcal{C})^{\op},\Ab)$.
\end{enumerate}
\end{theor}

Now we are in a position to deduce a characterization of absolutely pure objects in finitely accessible additive
categories, which generalizes \cite[Proposition~5.6]{Prest} from finitely accessible abelian categories to finitely
accessible additive categories. Note that every object of a finitely accessible additive category does have a
pure-injective envelope \cite[Theorem~6]{Herzog}.

\begin{cor} \cite[Theorem~3.3]{C14} \label{c:charfac} Let $\mathcal{C}$ be a finitely accessible additive category. Then
the following are equivalent for an object $X$ of $\mathcal{C}$:
\begin{enumerate}[(i)]
 \item $X$ is absolutely pure.
 \item The pure-injective envelope of $X$ is injective.
 \item $X$ is a pure subobject of an injective object.
 \item $X$ is a pure subobject of an absolutely pure object.
\end{enumerate}
\end{cor}

\begin{proof} This follows by Proposition \ref{p:charmod} and Theorem \ref{t:equivclasses}, using the fact that purity
is preserved and reflected by the covariant Yoneda functor $H:\mathcal{C}\to (\fp(\mathcal{C})^{\op},\Ab)$.
\end{proof}

\section{Applications to module categories}

Throughout this section ${\rm Mod}(R)$ denotes the category of right modules over a ring $R$ with identity. We may
consider projectively generated and injectively generated exact structures $\mathcal{E}_H^{\mathcal{M}}$ and
$\mathcal{E}_G^{\mathcal{M}}$ (see \cite[Proposition~2.6]{CK1}). Note that short exact sequences in
$\mathcal{E}_H^{\mathcal{M}}$ and $\mathcal{E}_G^{\mathcal{M}}$ are called $\mathcal{M}$-{\it pure}
and $\mathcal{M}$-{\it copure} (see \cite[Chapter~7]{Wisb}). 

One may also consider another particular generated exact structure on ${\rm Mod}(R)$ (e.g., see
\cite[Section~4]{Montano} and \cite{St67}), as follows.

\begin{prop} \label{p:flatgen} For any right $R$-module $M$, consider the additive functor $T=M\otimes_R-:{\rm
Mod}(R^{\rm op})\to {\rm Ab}$. For a class $\mathcal{M}$ of right $R$-modules, denote
$$\mathcal{E}_T^{\mathcal{M}}=\bigcap \{\mathcal{E}_T \mid M\in \mathcal{M}\}.$$ Then $\mathcal{E}_T^{\mathcal{M}}$
is an exact structure on ${\rm Mod}(R^{\rm op})$, called the exact structure {\it flatly generated by $\mathcal{M}$}.
\end{prop}

\begin{prop} \label{p:Tor} Let $\mathcal{E}$ be an exact structure on ${\rm Mod}(R^{\rm op})$ such that $\mathcal{E}$ is
flatly generated by a class $\mathcal{M}$ of right $R$-modules. Then a left $R$-module $Z$ is $\mathcal{E}$-flat if and
only if ${\rm Tor}_1^R(\mathcal{M},Z)=0$.  
\end{prop}

\begin{proof} Consider a short exact sequence $0\to X\to Y\to Z\to 0$ of left $R$-modules
with $Y$ projective. For every $M\in \mathcal{M}$, it induces the exact sequence \[0={\rm Tor}_1^R(M,Y)\to {\rm
Tor}_1^R(M,Z)\to M\otimes_R X\to M\otimes_R Y\to M\otimes_R Z\to 0\]

Assume first that $Z$ is $\mathcal{E}$-flat. Then $0\to X\to Y\to Z\to 0$ is an $\mathcal{E}$-conflation. Since
$\mathcal{E}=\mathcal{E}^{\mathcal{M}}_T$, it induces a short exact sequence $0\to M\otimes_R X\to M\otimes_R Y\to
M\otimes_R Z\to 0$ for every $M\in \mathcal{M}$. From the above exact sequence, it follows that ${\rm
Tor}_1^R(M,Z)=0$ for every $M\in \mathcal{M}$. 

Conversely, assume that ${\rm Tor}_1^R(\mathcal{M},Z)=0$. Then we have the short exact sequence $0\to
M\otimes_R X\to M\otimes_R Y\to M\otimes_R Z\to 0$ for every $M\in \mathcal{M}$. This shows that the short exact
sequence $0\to X\to Y\to Z\to 0$ is an $\mathcal{E}^{\mathcal{M}}_T$-conflation. This implies that $Z$ is 
$\mathcal{E}$-flat.
\end{proof}

Let us see some of the most important examples of exact structures on module categories.

\begin{exam} \label{e:exact} \rm (1) Let $\mathcal{M}$ be the class of finitely presented right $R$-modules. 

The $\mathcal{M}$-pure short exact sequences in $\Mod(R)$ are called \emph{pure} \cite[p.~281]{Wisb}. Let $\mathcal{D}$
and $\mathcal{E}$ be the exact structures given by short exact sequences and pure short exact sequences in ${\rm
Mod}(R)$ respectively. Then $\mathcal{D}$-$\mathcal{E}$-divisible and $\mathcal{D}$-$\mathcal{E}$-flat objects coincide
with absolutely pure and flat right $R$-modules respectively. Note that the exact structure of pure short exact
sequences is also injectively generated by the class of pure-injective modules \cite[34.7]{Wisb}.

The $\mathcal{M}$-copure short exact sequences in $\Mod(R)$ are called \emph{copure} (note the different terminology
with respect to \cite[p.~322]{Wisb}, where $\mathcal{M}$ is the class of finitely copresented right $R$-modules). Let
$\mathcal{D}$ and $\mathcal{E}$ be the exact structures given by short exact sequences and copure short exact sequences
in ${\rm Mod}(R)$ respectively. Then $\mathcal{D}$-$\mathcal{E}$-divisible and $\mathcal{D}$-$\mathcal{E}$-flat objects
will be called \emph{absolutely copure} and \emph{coflat} right $R$-modules respectively.

The short exact sequences in $\Mod(R^{\rm op})$ which give the exact structure flatly generated by the
class of finitely presented (or all) right $R$-modules coincide with pure short exact sequences \cite[34.5]{Wisb}.

(2) Let $\mathcal{M}$ be the class of finitely generated right $R$-modules. 

The $\mathcal{M}$-pure short exact sequences in $\Mod(R)$ are called \emph{finitely split} \cite{Az87}. Let
$\mathcal{D}$ and $\mathcal{E}$ be the exact structures given by pure (all) short exact sequences and
finitely split short exact sequences in ${\rm Mod}(R)$ respectively. Then $\mathcal{D}$-$\mathcal{E}$-divisible and
$\mathcal{D}$-$\mathcal{E}$-flat objects coincide with finitely pure-injective (finitely injective) and finitely
pure-projective (finitely projective) right $R$-modules respectively \cite[p.~115]{Az87}.

The $\mathcal{M}$-copure short exact sequences in $\Mod(R)$ will be called \emph{cofinitely split}. Let $\mathcal{D}$
and $\mathcal{E}$ be the exact structures given by pure (all) short exact sequences and cofinitely split short exact
sequences in ${\rm Mod}(R)$ respectively. Then $\mathcal{D}$-$\mathcal{E}$-divisible and
$\mathcal{D}$-$\mathcal{E}$-flat objects will be called \emph{cofinitely pure-injective} (\emph{cofinitely injective})
and \emph{cofinitely pure-projective} (\emph{cofinitely projective}) right $R$-modules respectively.

The short exact sequences in $\Mod(R^{\rm op})$ which give the exact structure flatly generated by the
class of finitely generated (or all) right $R$-modules coincide with pure short exact sequences \cite[34.5]{Wisb}.

(3) Let $\mathcal{M}$ be the class of (semi)simple modules right $R$-modules. 

The $\mathcal{M}$-pure short exact sequences in $\Mod(R)$ are called \emph{neat} \cite{Fuc}. Let $\mathcal{D}$ and
$\mathcal{E}$ be the exact structures given by short exact sequences and neat short exact sequences in ${\rm Mod}(R)$
respectively. Then $\mathcal{D}$-$\mathcal{E}$-divisible and $\mathcal{D}$-$\mathcal{E}$-flat objects coincide
with absolutely neat \cite{Sep} and neat-flat \cite{BD15} right $R$-modules respectively. Note that by
\cite[Theorem~3.4]{Sep} absolutely neat right $R$-modules coincide with $m$-injective right $R$-modules in the sense of
\cite{S98}, that is, right $R$-modules which are injective with respect to every short exact sequence of the form $0\to
I\to R\to R/I\to 0$ with $I$ a maximal right ideal of $R$.  

The $\mathcal{M}$-copure short exact sequences in $\Mod(R)$ are called \emph{coneat} by Fuchs \cite{Fuc}. Let
$\mathcal{D}$ and $\mathcal{E}$ be the exact structures given by short exact sequences and coneat short exact sequences
in ${\rm Mod}(R)$ respectively. Then $\mathcal{D}$-$\mathcal{E}$-divisible and $\mathcal{D}$-$\mathcal{E}$-flat
objects coincide with absolutely coneat \cite{Sep} and coneat-flat \cite{BD14} right $R$-modules respectively.

The short exact sequences in $\Mod(R^{\rm op})$ which give the exact structure flatly generated by $\mathcal{M}$ are
called \emph{$s$-pure} \cite{CC}. Let $\mathcal{D}$ and $\mathcal{E}$ be the exact structures given by short exact
sequences and $s$-pure short exact sequences in ${\rm Mod}(R^{\rm op})$ respectively. Then
$\mathcal{D}$-$\mathcal{E}$-divisible and $\mathcal{D}$-$\mathcal{E}$-flat objects coincide with absolutely
$s$-pure \cite{CC} and max-flat \cite{Xiang} left $R$-modules respectively. Note that if $R$ is commutative, then coneat
short exact sequences in the sense of Fuchs are the same as $s$-pure short exact sequences \cite[Proposition~3.1]{Fuc}.

(4) Let $\mathcal{M}$ be the class of right $R$-modules $M$ with Jacobson radical $\Rad (M)=0$. 

The $\mathcal{M}$-pure short exact sequences in $\Mod(R)$ will be called \emph{rad-neat}. Let $\mathcal{D}$ and
$\mathcal{E}$ be the exact structures given by short exact sequences and rad-neat short exact sequences in ${\rm
Mod}(R)$ respectively. Then $\mathcal{D}$-$\mathcal{E}$-divisible and $\mathcal{D}$-$\mathcal{E}$-flat objects will be
called \emph{absolutely rad-neat} and \emph{rad-neat-flat} right $R$-modules respectively.

The $\mathcal{M}$-copure short exact sequences in $\Mod(R)$ will be called \emph{rad-coneat}. They were studied under
the name of coneat short exact sequences by Mermut \cite{Engin}. Let $\mathcal{D}$ and $\mathcal{E}$ be the exact
structures given by short exact sequences and rad-coneat short exact sequences in ${\rm Mod}(R)$ respectively. Then
$\mathcal{D}$-$\mathcal{E}$-divisible and $\mathcal{D}$-$\mathcal{E}$-flat objects will be called \emph{absolutely
rad-coneat} and \emph{rad-coneat-flat} right $R$-modules respectively. Note that if $R$ is semilocal,
then coneat short exact sequences are the same as rad-coneat short exact sequences \cite[Theorem 3.8.7]{Engin}.

The short exact sequences in $\Mod(R^{\rm op})$ which give the exact structure flatly generated by $\mathcal{M}$ will be
called \emph{rad-pure} in this paper. Let $\mathcal{D}$ and $\mathcal{E}$ be the exact structures given by short exact
sequences and rad-pure short exact sequences in $\Mod(R^{\rm op})$ respectively. Then
$\mathcal{D}$-$\mathcal{E}$-divisible and $\mathcal{D}$-$\mathcal{E}$-flat objects will be called \emph{absolutely
rad-pure} and \emph{rad-pure-flat} left $R$-modules respectively. Note that if $R$ is commutative semilocal, then coneat
(equivalently, rad-coneat) short exact sequences are the same as rad-pure short exact sequences.
\end{exam}

\begin{remark} \rm Our next results and applications will refer to the above examples of exact structures and to the
case when $\mathcal{D}$ is the exact structure on module categories given by all short exact sequences. We show
some properties relating relative divisibility and relative flatness with absolute purity and flatness respectively.
But with some exceptions of less known or unknown results, we do not intend to deduce corollaries for the pure exact
structure on module categories, which has been extensively studied. We rather insist on exact structures generated by
the simple modules and the modules with zero Jacobson radical in the next sections. There are also some other relevant
examples of exact structures in module categories, for which the interested reader may obtain consequences of our
properties. For instance, complement (or closed), supplement and coclosed submodules induce exact structures on module
categories (e.g., see \cite[Section~5]{Montano}). 
\end{remark}

For a left $R$-module $X$, we denote by $X^+={\rm Hom}_{\mathbb{Z}}(X,\mathbb{Q}/\mathbb{Z})$ its character module,
which is a right $R$-module. For a short exact sequence $E: 0\to X\to Y\to Z\to 0$ of left $R$-modules, we denote by
$E^+: 0\to Z^+\to Y^+\to X^+\to 0$ the induced short exact sequence of right $R$-modules. 

The following property shows that we may study flatly generated exact structures in module categories by means of
projectively generated exact structures. 

\begin{theor} \cite[Theorem~8.1]{Skly} Let $\mathcal{M}$ be a class of right $R$-modules and $E$ a short exact sequence
of left $R$-modules. Then $E\in \mathcal{E}_T^{\mathcal{M}}$ if and only if $E^+\in \mathcal{E}_H^{\mathcal{M}}$.  
\end{theor}

\begin{prop} \label{p:char} Let $\mathcal{M}$ be a class of right $R$-modules and $Z$ a left $R$-module. Then $Z$ is
$\mathcal{E}_T^{\mathcal{M}}$-flat if and only if $Z^+$ is an $\mathcal{E}_H^{\mathcal{M}}$-divisible right $R$-module. 
\end{prop}

\begin{proof} Use \cite[Proposition~3.3]{CK1}, Proposition \ref{p:Tor} and ${\rm Ext}^1_R(\mathcal{M},Z^+)\cong
{\rm Tor}^R_1(\mathcal{M},Z)^+$.
\end{proof}

\begin{cor} \label{c:puresub} Let $\mathcal{M}$ be a class of right $R$-modules. Then the class of
$\mathcal{E}_T^{\mathcal{M}}$-flat left $R$-modules is closed under pure submodules. 
\end{cor}

\begin{proof} Let $Z$ be an $\mathcal{E}_T^{\mathcal{M}}$-flat left $R$-module and let $Y$ be a pure submodule of $Z$.
Then we have an induced split epimorphism $Z^+\to Y^+$. Since $Z^+$ is an $\mathcal{E}_H^{\mathcal{M}}$-divisible right
$R$-module by Proposition \ref{p:char}, so is $Y^+$. Then $Y$ is $\mathcal{E}_T^{\mathcal{M}}$-flat again by
Proposition \ref{p:char}. 
\end{proof}

Recall that a right $R$-module $Z$ is called \emph{$FP$-projective} if ${\rm Ext}^1_R(Z,X)=0$ for every absolutely pure
right $R$-module $X$ \cite{Mao13}. 

\begin{prop} \label{p:Mfpproj} Let $\mathcal{E}$ be an exact structure on ${\rm Mod}(R)$.
\begin{enumerate}
\item Assume that $\mathcal{E}$ is projectively generated by a class $\mathcal{M}$ of right $R$-modules. Then every
right $R$-module in $\mathcal{M}$ is $FP$-projective if and only if every absolutely pure right $R$-module is
$\mathcal{E}$-divisible. 
\item Assume that $\mathcal{E}$ is injectively generated by a class $\mathcal{M}$ of right $R$-modules. Then every
right $R$-module in $\mathcal{M}$ is cotorsion if and only if every flat right $R$-module is $\mathcal{E}$-flat. 
\end{enumerate}
\end{prop}

\begin{proof} (1) Note that $(\mathcal{A},\mathcal{B})$ is a cotorsion pair in ${\rm Mod}(R)$ for $\mathcal{A}$ and
$\mathcal{B}$ the classes of $FP$-projective and absolutely pure right $R$-modules respectively
\cite[Theorem~2.14]{Mao13}.

Assume first that every module in $\mathcal{M}$ is $FP$-projective. Let $X$ be an absolutely pure
right $R$-module. Then ${\rm Ext}^1_R(F,X)=0$ for every $FP$-projective right $R$-module $F$, hence ${\rm
Ext}^1_R(M,X)=0$ for every $M\in \mathcal{M}$. Then $X$ is $\mathcal{E}$-divisible by \cite[Proposition~3.3]{CK1}.
 
Conversely, assume that every absolutely pure right $R$-module is $\mathcal{E}$-divisible. Let $M\in \mathcal{M}$. By
hypothesis and \cite[Proposition~3.3]{CK1}, ${\rm Ext}^1_R(M,X)=0$ for every absolutely pure right $R$-module $X$. Then $M$
is $FP$-projective.

(2) Note that $(\mathcal{A},\mathcal{B})$ is a cotorsion pair in ${\rm Mod}(R)$ for $\mathcal{A}$ and $\mathcal{B}$ the
classes of flat and cotorsion right $R$-modules respectively \cite[Theorem~3]{BEE}, and proceed in a similar way as for
(1).
\end{proof}

\begin{cor} \label{c:Mfp} Let $\mathcal{E}$ be an exact structure on ${\rm Mod}(R)$ projectively generated by a class
$\mathcal{M}$ of finitely generated right $R$-modules. Then every right $R$-module in $\mathcal{M}$ is finitely
presented if and only if every absolutely pure right $R$-module is $\mathcal{E}$-divisible. 
\end{cor}

\begin{proof} Note that a finitely generated right $R$-module is $FP$-projective if and only if it is finitely presented
\cite[Proposition]{Enochs76}. 
\end{proof}

\begin{prop} \label{p:Mfgp} Let $\mathcal{E}$ be an exact structure on ${\rm Mod}(R)$ projectively generated by a class
$\mathcal{M}$ of finitely generated right $R$-modules. Then the following are equivalent:
\begin{enumerate}[(i)]
\item Every right $R$-module from $\mathcal{M}$ embeds into a (finitely generated) projective right $R$-module.
\item Every absolutely pure right $R$-module is $\mathcal{E}$-flat.
\item Every injective right $R$-module is $\mathcal{E}$-flat.
\item The injective envelope of every right $R$-module from $\mathcal{M}$ is $\mathcal{E}$-flat.
\item For every free left $R$-module $Z$, $Z^+$ is an $\mathcal{E}$-flat right $R$-module.  
\end{enumerate}
\end{prop}

\begin{proof} (i)$\Rightarrow$(ii) Let $X$ be an absolutely pure right $R$-module. Let $f:M\to E$ be a homomorphism
with $M\in \mathcal{M}$. By hypothesis, $M$ embeds into some finitely generated projective right $R$-module $P$. Since
$P/M$ is finitely presented, $f$ can be extended to a homomorphism $P\to X$. Then $X$ is $\mathcal{E}$-flat by
\cite[Proposition~4.6]{CK1}.

(ii)$\Rightarrow$(iii) and (iii)$\Rightarrow$(iv) These are clear.

(iv)$\Rightarrow$(i) Let $M\in \mathcal{M}$. Since the injective envelope $E(M)$ of $M$ is $\mathcal{E}$-flat, the
inclusion homomorphism $i:M\to E(M)$ factors through a projective right $R$-module $P$ by \cite[Proposition~4.6]{CK1}.
It follows that $M$ embeds into $P$. 

(iii)$\Rightarrow$(v) Let $Z$ be a free left $R$-module. Then $Z^+$ is an injective right $R$-module, hence $Z^+$ is
$\mathcal{E}$-flat by hypothesis. 

(v)$\Rightarrow$(iii) Let $X$ be an injective right $R$-module. Consider an epimorphism $Y\to X^+$ for some free left
$R$-module $Y$. Then $Y^+$ is $\mathcal{E}$-flat by hypothesis and we have an induced monomorphism $X\to X^{++}\to
Y^+$. But this splits, since $X$ is injective, hence $X$ is $\mathcal{E}$-flat. 
\end{proof}

Recall that a ring $R$ is called a \emph{right IF ring} if every injective right $R$-module is flat \cite{Colby}.

\begin{cor} \cite[Theorem~1]{Colby} The following are equivalent:
\begin{enumerate}[(i)]
\item Every finitely presented right $R$-module embeds into a (finitely generated) projective right $R$-module.
\item Every absolutely pure right $R$-module is flat.
\item $R$ is a right IF ring.
\item The injective envelope of every finitely presented right $R$-module is flat.
\item For every free left $R$-module $Z$, $Z^+$ is a flat right $R$-module.  
\end{enumerate}
\end{cor}

\begin{proof} This follows by Proposition \ref{p:Mfgp}, considering the exact structure given by pure short exact
sequences of right $R$-modules. 
\end{proof}

Recall that a ring $R$ is called a \emph{right FGF ring} if every finitely generated right $R$-module embeds into a free
(projective) right $R$-module. 

\begin{cor} The following are equivalent:
\begin{enumerate}[(i)]
\item $R$ is a right FGF ring.
\item Every absolutely pure right $R$-module is finitely projective.
\item Every injective right $R$-module is finitely projective.
\item The injective envelope of every finitely generated right $R$-module is finitely projective.
\item For every free left $R$-module $Z$, $Z^+$ is a finitely projective right $R$-module.  
\end{enumerate}
\end{cor}

\begin{proof} This follows by Proposition \ref{p:Mfgp}, considering the exact structure given by finitely split short
exact sequences of right $R$-modules. 
\end{proof}

\begin{cor} \label{c:fproj} Let $\mathcal{E}$ be an exact structure on ${\rm Mod}(R)$. 
\begin{enumerate}
\item Assume that $\mathcal{E}$ is projectively generated by a class $\mathcal{M}$ of right $R$-modules. Then the
following are equivalent:
\begin{enumerate}[(i)]
\item For every short exact sequence $0\to X\to Y'\to Z'\to 0$ of right $R$-modules with $X$ absolutely pure
(respectively cotorsion) and $Y'$ $\mathcal{E}$-divisible, $Z'$ is $\mathcal{E}$-divisible.
\item For every short exact sequence $0\to Z\to U\to V\to 0$ of right $R$-modules with $V\in \mathcal{M}$ and $U$
projective, $Z$ is $FP$-projective (respectively flat).
\end{enumerate}
\item Assume that $\mathcal{E}$ is injectively generated by a class $\mathcal{M}$ of right $R$-modules. Then the
following are equivalent:
\begin{enumerate}[(i)]
\item For every short exact sequence $0\to Z\to U\to V\to 0$ of right $R$-modules with $V$ $FP$-projective (respectively
flat) and $U$ $\mathcal{E}$-flat, $Z$ is $\mathcal{E}$-flat.
\item For every short exact sequence $0\to X\to Y'\to Z'\to 0$ of right $R$-modules with $X\in \mathcal{M}$ and
$Y'$ injective, $Z'$ is absolutely pure (respectively cotorsion).
\end{enumerate}
\end{enumerate}
\end{cor}

\begin{proof} Use \cite[Proposition~4.3]{CK1} for the cotorsion pair $(\mathcal{A},\mathcal{B})$, where $\mathcal{A}$ is the
class of $FP$-projective (respectively flat) right $R$-modules and $\mathcal{B}$ is the class of absolutely pure
(respectively cotorsion) right $R$-modules \cite[Theorem~3]{BEE}, \cite[Theorem~2.14]{Mao13}. 
\end{proof}

\begin{cor} \label{c:quot} Let $\mathcal{E}$ be an exact structure on ${\rm Mod}(R)$. 
\begin{enumerate}
\item Assume that $\mathcal{E}$ is projectively generated by a class $\mathcal{M}$ of right $R$-modules. Then the
following are equivalent:
\begin{enumerate}[(i)]
\item The class of $\mathcal{E}$-divisible right $R$-modules is closed under homomorphic images.
\item For every short exact sequence $0\to Z\to U\to V\to 0$ of right $R$-modules with $V\in \mathcal{M}$ and $U$
projective, $Z$ is projective.
\end{enumerate}
\item Assume that $\mathcal{E}$ is injectively generated by a class $\mathcal{M}$ of right $R$-modules. Then the
following are equivalent:
\begin{enumerate}[(i)]
\item The class of $\mathcal{E}$-flat right $R$-modules is closed under submodules.
\item For every short exact sequence $0\to X\to Y'\to Z'\to 0$ of right $R$-modules with $X\in \mathcal{M}$ and
$Y'$ injective, $Z'$ is injective.
\end{enumerate}
\end{enumerate}
\end{cor}

\begin{proof} This follows by \cite[Proposition~4.4]{CK1}. 
\end{proof}

\begin{cor} 
\begin{enumerate}[(i)]
\item \cite[Theorem~3.2]{St70} The class of absolutely pure right $R$-modules is coresolving if and only if $R$ is right coherent.
\item \cite[Theorem~2]{Megibben} The class of absolutely pure right $R$-modules is closed under homomorphic images if
and only if $R$ is right semihereditary.
\end{enumerate}
\end{cor}

\begin{proof} This follows by Corollaries \ref{c:fproj} and \ref{c:quot}, considering the exact structure given by pure
short exact sequences of right $R$-modules. 
\end{proof}

\begin{defi} \label{defcoh} \rm Let $\mathcal{N}$ be a class of right $R$-modules. A ring $R$ is called \emph{right
$\mathcal{N}$-coherent} if every $N\in \mathcal{N}$ is $2$-presented, that is, there is an exact sequence $F_2\to
F_1\to F_0\to N\to 0$ of right $R$-modules with $F_0,F_1,F_2$ finitely generated free.
\end{defi}

\begin{prop} \label{p:coh} Let $\mathcal{N}$ be a class of finitely presented right $R$-modules. Then the following are
equivalent: 
\begin{enumerate}[(i)] 
\item $R$ is right $\mathcal{N}$-coherent.
\item $\underset{\rightarrow}\lim {\rm Ext}^1_R(N,X_i)\cong {\rm Ext}^1_R(N,\underset{\rightarrow}\lim X_i)$ for every
$N\in \mathcal{N}$ and every direct system $(X_i)_{i\in I}$ of right $R$-modules.
\item The class of $\mathcal{E}_H^{\mathcal{N}}$-divisible right $R$-modules is closed under direct limits.
\item The class of $\mathcal{E}_H^{\mathcal{N}}$-divisible right $R$-modules is closed under pure quotients.
\item ${\rm Tor}_1^R(N,\prod_{i\in I}Z_i)\cong \prod_{i\in I} {\rm Tor}_1^R(N,Z_i)$ for every $N\in \mathcal{N}$ and
every family $(Z_i)_{i\in I}$ of left $R$-modules.
\item The class of $\mathcal{E}_T^{\mathcal{N}}$-flat left $R$-modules is closed under direct products.
\item A right $R$-module $X$ is $\mathcal{E}_H^{\mathcal{N}}$-divisible if and only if $X^+$ is
an $\mathcal{E}_T^{\mathcal{N}}$-flat left $R$-module.
\item A right $R$-module $X$ is $\mathcal{E}_H^{\mathcal{N}}$-divisible if and only if $X^{++}$ is an
$\mathcal{E}_H^{\mathcal{N}}$-divisible right $R$-module. 
\item A left $R$-module $Z$ is $\mathcal{E}_T^{\mathcal{N}}$-flat if and only if $Z^{++}$ is an 
$\mathcal{E}_T^{\mathcal{N}}$-flat left $R$-module.
\end{enumerate}  
\end{prop}

\begin{proof} (i)$\Rightarrow$(ii) This follows by \cite[Lemma~2.9]{CD}, because every $N\in \mathcal{N}$ is
$2$-presented.

(ii)$\Rightarrow$(iii) This is clear.

(iii)$\Rightarrow$(i) Let $N\in \mathcal{N}$. Since $N$ is finitely presented, there is an exact sequence $0\to K\to
M\to N\to 0$ of right $R$-modules with $M$ finitely generated free and $K$ finitely generated. Let $(X_i)_{i\in I}$
be a direct system of injective right $R$-modules. Then $\underset{\rightarrow} \lim X_i$ is
$\mathcal{E}_H^{\mathcal{N}}$-divisible, and so ${\rm Ext}^1_R(N,\underset{\rightarrow} \lim X_i)=0$ 
by \cite[Proposition~3.3]{CK1}. Then we have the following induced commutative diagram with exact rows:
\[\SelectTips{cm}{}
\xymatrix{
0 \ar[r] & {\rm Hom}_R(N,\underset{\rightarrow} \lim X_i) \ar[d]_f \ar[r] & {\rm Hom}_R(M,\underset{\rightarrow} \lim
X_i) \ar[d]^g \ar[r] & {\rm Hom}_R(K,\underset{\rightarrow} \lim X_i) \ar[d]^h \ar[r] & 0 \\ 0 \ar[r] &
\underset{\rightarrow} \lim {\rm Hom}_R(N,X_i) \ar[r] & \underset{\rightarrow} \lim {\rm Hom}_R(M,X_i)
\ar[r] & \underset{\rightarrow} \lim {\rm Hom}_R(K,X_i) \ar[r] & 0
}\]
Since $M$ and $N$ are finitely presented, $f$ and $g$ are isomorphisms. It follows that $h$ is an isomorphism, and so
$K$ is finitely presented. Then $N$ is $2$-presented, which shows that $R$ is right $\mathcal{N}$-coherent.

(i)$\Rightarrow$(v) This follows by \cite[Lemma~2.10]{CD}, because every $N\in \mathcal{N}$ is $2$-presented.

(v)$\Rightarrow$(vi) This is clear.

(vi)$\Rightarrow$(i) Let $N\in \mathcal{N}$. Since $N$ is finitely presented, there is an exact sequence $0\to K\to M\to
N\to 0$ of right $R$-modules with $M$ finitely generated free and $K$ finitely generated. Let $(Z_i)_{i\in I}$ be a
family of projective left $R$-modules. Then $\prod_{i\in I} Z_i$ is $\mathcal{E}_T^{\mathcal{N}}$-flat, and so 
${\rm Tor}_1^R(N,\prod_{i\in I} Z_i)=0$ by Proposition \ref{p:Tor}. Then we have the following induced commutative
diagram with exact rows: 
\[\SelectTips{cm}{}
\xymatrix{
0 \ar[r] & K\otimes_R (\prod_{i\in I} Z_i) \ar[d]_f \ar[r] & M\otimes_R (\prod_{i\in I} Z_i) \ar[d]^g \ar[r] &
N\otimes_R (\prod_{i\in I} Z_i) \ar[d]^h \ar[r] & 0 \\
0 \ar[r] & \prod_{i\in I} (K\otimes_R Z_i) \ar[r] & \prod_{i\in I} (M\otimes_R Z_i) \ar[r] & \prod_{i\in I}
(N\otimes_R Z_i) \ar[r] & 0
}\]
Since $M$ and $N$ are finitely presented, $g$ and $h$ are isomorphisms. It follows that $f$ is an isomorphism, and so
$K$ is finitely presented. Then $N$ is $2$-presented, which shows that $R$ is right $\mathcal{N}$-coherent.

(i)$\Rightarrow$(vii) For every $N\in \mathcal{N}$ we have ${\rm Tor}_1^R(N,X^+)\cong {\rm Ext}^1_R(N,X)^+$, because $N$
is $2$-presented. Then use \cite[Proposition~3.3]{CK1} and Proposition \ref{p:Tor}.

(vii)$\Rightarrow$(viii) This is clear, using Proposition \ref{p:char}.

(viii)$\Rightarrow$(ix) Let $Z$ be a left $R$-module. If $Z$ is $\mathcal{E}_T^{\mathcal{N}}$-flat, then $Z^+$ is an
$\mathcal{E}_H^{\mathcal{N}}$-divisible right $R$-module by Proposition \ref{p:char}. It follows that $Z^{+++}$ is an
$\mathcal{E}_H^{\mathcal{N}}$-divisible right $R$-module by hypothesis, and so $Z^{++}$ is an
$\mathcal{E}_T^{\mathcal{N}}$-flat left $R$-module again by Proposition \ref{p:char}. Conversely, if $Z^{++}$ is
$\mathcal{E}_T^{\mathcal{N}}$-flat, then $Z$ is also $\mathcal{E}_T^{\mathcal{N}}$-flat by Corollary \ref{c:puresub},
because the inclusion $Z\to Z^{++}$ is a pure monomorphism.

(ix)$\Rightarrow$(vi) Let $(Z_i)_{i\in I}$ be a family of $\mathcal{E}_T^{\mathcal{N}}$-flat left $R$-modules. By
\cite[Proposition~4.1]{CK1}, $\bigoplus_{i\in I}Z_i$ is $\mathcal{E}_T^{\mathcal{N}}$-flat. We have $\left (
\prod_{i\in I} Z_i^+\right )^+\cong \left (\bigoplus_{i\in I}Z_i\right )^{++}$, which is
$\mathcal{E}_T^{\mathcal{N}}$-flat by hypothesis. Since $\bigoplus_{i\in I}M_i^+$ is a pure submodule of
$\prod_{i\in I}M_i^+$, $\left(\bigoplus_{i\in I}Z_i^+\right)^+$ is isomorphic to a direct summand of
$\left(\prod_{i\in I}Z_i^+\right)^+$. Hence $\left(\bigoplus_{i\in I}Z_i^+\right)^+$ is
$\mathcal{E}_T^{\mathcal{N}}$-flat. For every $i\in I$, the inclusion $Z_i\to Z_i^{++}$ is a pure monomorphism,
hence so is the induced inclusion $\prod_{i\in I}Z_i\to \prod_{i\in I}Z_i^{++}$. But $\prod_{i\in I}Z_i^{++}\cong
\left(\bigoplus_{i\in I}Z_i^+\right)^+$, which is $\mathcal{E}_T^{\mathcal{N}}$-flat. Finally, $\prod_{i\in I}Z_i$ is
$\mathcal{E}_T^{\mathcal{N}}$-flat by Corollary \ref{c:puresub}.

(iv)$\Rightarrow$(iii) Let $(X_i)_{i\in I}$ be a direct system of $\mathcal{E}_H^{\mathcal{N}}$-divisible right
$R$-modules. Then $\bigoplus_{i\in I}X_i$ is $\mathcal{E}_H^{\mathcal{N}}$-divisible by \cite[Proposition~4.1]{CK1}.
Now the canonical pure epimorphism $\bigoplus_{i\in I}X_i\to \underset{\rightarrow} \lim X_i$ implies that
$\underset{\rightarrow} \lim X_i$ is $\mathcal{E}_H^{\mathcal{N}}$-divisible.

(vii)$\Rightarrow$(iv) Let $X$ be an $\mathcal{E}_H^{\mathcal{N}}$-divisible right $R$-module and $X\to Y$ a pure
epimorphism. Then $X^+$ is an $\mathcal{E}_T^{\mathcal{N}}$-flat left $R$-module and the induced monomorphism $Y^+\to
X^+$ splits. Hence $Y^+$ is $\mathcal{E}_T^{\mathcal{N}}$-flat by \cite[Proposition~4.1]{CK1}, and so $Y$ is
$\mathcal{E}_H^{\mathcal{N}}$-divisible.
\end{proof}

\begin{remark} \rm Note that if $\mathcal{N}$ is the class of all finitely presented right $R$-modules, then $R$ is
right $\mathcal{N}$-coherent if and only if $R$ is right coherent. In this case Proposition \ref{p:coh} gives a
well-known characterization of right coherent rings in terms of absolutely pure right $R$-modules and flat left
$R$-modules (e.g., see \cite[26.6, 35.8]{Wisb}). 
\end{remark}

\begin{prop} \label{p:coh2} Let $\mathcal{N}$ be a class of finitely presented right $R$-modules. Then the following are
equivalent: 
\begin{enumerate}[(i)] 
\item $R$ is right $\mathcal{N}$-coherent and every $\mathcal{E}_T^{\mathcal{N}}$-flat left $R$-module is flat.
\item $R$ is right coherent and a right $R$-module $X$ is $\mathcal{E}_H^{\mathcal{N}}$-divisible if and only if $X$ is
absolutely pure.
\item A right $R$-module $X$ is $\mathcal{E}_H^{\mathcal{N}}$-divisible if and only if $X^+$ is a flat left $R$-module.
\end{enumerate}  
\end{prop}

\begin{proof} (i)$\Rightarrow$(ii) Let $X$ be a right $R$-module. Assume that $X$ is
$\mathcal{E}_H^{\mathcal{N}}$-divisible. Since $R$ is right $\mathcal{N}$-coherent, $X^+$ is
$\mathcal{E}_T^{\mathcal{N}}$-flat by Proposition \ref{p:coh}. Then $X^+$ is flat by hypothesis, and so $X^{++}$ is
injective. But the inclusion $X\to X^{++}$ is a pure monomorphism, whence it follows that $X$ is absolutely pure.
Conversely, assume that $X$ is absolutely pure. Since $\mathcal{N}$ is a class of finitely presented right $R$-modules,
it follows that ${\rm Ext}^1_R(\mathcal{N},X)=0$. Then $X$ is $\mathcal{E}_H^{\mathcal{N}}$-divisible by 
\cite[Proposition~3.3]{CK1}. Finally, using again Proposition \ref{p:coh}, $X$ is absolutely pure if and only if $X$ is
$\mathcal{E}_H^{\mathcal{N}}$-divisible if and only if $X^+$ is $\mathcal{E}_T^{\mathcal{N}}$-flat if and only if $X^+$
is flat. Then $R$ is right coherent. 

(ii)$\Rightarrow$(i) Since $R$ is right coherent, every finitely presented right $R$-module is $2$-presented, hence $R$
is right $\mathcal{N}$-coherent. Now let $Z$ be an $\mathcal{E}_T^{\mathcal{N}}$-flat left $R$-module. Then $Z^+$ is an
$\mathcal{E}_H^{\mathcal{N}}$-divisible left $R$-module by Proposition \ref{p:char}, and so $Z^+$ is absolutely pure by
hypothesis. But this implies that $Z$ is flat. 

(ii)$\Rightarrow$(iii) This is clear.

(iii)$\Rightarrow$(ii) In order to show that $R$ is right coherent, it is enough to prove that a left $R$-module $Z$ is
flat if and only if $Z^{++}$ is flat. To this end, let $Z$ be a left $R$-module. If $Z$ is flat, then $Z^+$ is
$\mathcal{E}_H^{\mathcal{N}}$-divisible by Proposition \ref{p:char}, and so $Z^{++}$ is flat by hypothesis. Conversely,
if $Z^{++}$ is flat, then $Z$ is flat, because the inclusion $Z\to Z^{++}$ is a pure monomorphism and the class of flat
left $R$-modules is closed under pure submodules. Hence $R$ is right coherent. The last part of (ii) is now clear.
\end{proof}

We leave the interested reader to deduce consequences of the above results in the case of the exact structure given
by all pure short exact sequences of modules. We end this section with some comments on pure Xu exact structures on
module categories.

\begin{remark} \rm Let $\mathcal{E}$ be the exact structure given by pure short exact sequences in ${\rm Mod}(R)$. Then
$\mathcal{E}$ is both projectively and injectively generated (see Example \ref{e:exact}). With the notation from 
\cite[Section~3]{CK1}, we have $\Psi(\mathcal{E})=(\mathcal{A}_1,\mathcal{B}_1)$, 
where $\mathcal{A}_1$ and $\mathcal{B}_1$ are the classes
of $FP$-projective and absolutely pure right $R$-modules respectively, and
$\Phi(\mathcal{E})=(\mathcal{A}_2,\mathcal{B}_2)$, where $\mathcal{A}_2$ and $\mathcal{B}_2$ are the classes of flat and
cotorsion right $R$-modules respectively. Then $\mathcal{E}$ is a projectively generated Xu exact structure if and only
if every $FP$-projective right $R$-module is pure-projective. Note that the hypothesis from \cite[Theorem~3.10]{CK1}
that every right $R$-module has a pure-projective cover is equivalent to $R$ being a right pure-semisimple ring (e.g.,
see \cite[Theorem~6.18]{AH}). In this case, every right $R$-module is pure-projective, and so $\mathcal{E}$ is clearly a
projectively generated Xu exact structure. On the other hand, over an arbitrary ring $R$, every right $R$-module has a
pure-injective envelope. Then by \cite[Theorem~3.10]{CK1}, $\mathcal{E}$ is an injectively generated Xu
exact structure if and only if every cotorsion right $R$-module is pure-injective if and only if the class of
pure-injective modules is closed under extensions, that is, $R$ is a right Xu ring \cite{HR}. 
\end{remark}

\section{Exact structures generated by the simple modules}

In this section we deduce and complete a series of known results concerning relatively divisible and relatively flat
modules with respect to the exact structures projectively, injectively and flatly generated by the class of simple
modules. Some of them have been previously given only for modules over commutative rings.

\begin{cor} \label{c:closureDES} 
\begin{enumerate} \item 
\begin{enumerate}[(i)] 
\item The classes of absolutely neat (respectively absolutely coneat, absolutely $s$-pure) modules are closed
under extensions and neat (respectively coneat, $s$-pure) submodules. 
\item The class of absolutely neat modules is closed under direct products.
\item The classes of absolutely neat, absolutely coneat and absolutely $s$-pure modules are closed under direct
sums.
\item Every module is absolutely neat (respectively absolutely coneat, absolutely $s$-pure) if and only if
every short exact sequence of modules is neat (respectively coneat, $s$-pure). 
\item Every module is absolutely neat if and only if $R$ is semisimple. 
\end{enumerate}
\item 
\begin{enumerate}[(i)]
\item The classes of neat-flat (respectively coneat-flat, max-flat) modules are closed under extensions and
neat (respectively coneat, $s$-pure) quotients. 
\item The classes of neat-flat, coneat-flat and max-flat modules are closed under direct sums.
\item The class of max-flat modules is closed under direct limits. If every simple module is finitely presented, then
the class of neat-flat modules is closed under direct limits.
\item Every module is neat-flat (respectively coneat-flat, max-flat) if and only if every short exact sequence
of modules is neat (respectively coneat, $s$-pure).
\item Every right $R$-module is coneat-flat if and only if $R$ is a right $V$-ring.
\end{enumerate}
\end{enumerate}
\end{cor}

\begin{proof} This mainly follows by \cite[Proposition~4.1]{CK1}, considering the exact structures given by neat
(respectively coneat, $s$-pure) short exact sequences of modules, which are projectively (respectively
injectively, flatly) generated by the class of simple modules. We only add some details.

(1) (ii) Using the isomorphism $\Hom_R(S,\prod_{i\in I}M_i)\cong \prod_{i\in I}\Hom_R(S,M_i)$ for every (simple) module
$S$ and every family $(M_i)_{i\in I}$ of modules, it is straightforward to show that the class of neat short
exact sequences is closed under direct products. 

(iii) Note that every simple module $S$ is finitely generated. Hence we have an isomorphism $\Hom_R(S,\bigoplus_{i\in
I}M_i)\cong \bigoplus_{i\in I}\Hom_R(S,M_i)$ for every family $(M_i)_{i\in I}$ of modules, from which it is
straightforward (or use \cite[18.2]{Wisb}) to show that the class of neat short exact sequences is closed under direct
sums. One can easily show that the classes of coneat and $s$-pure short exact sequences are closed under direct sums by
\cite[16.2]{Wisb} and \cite[12.15]{Wisb} respectively. 

(v) By \cite[Proposition~4.1]{CK1}, every module is absolutely neat if and only if every simple module is
projective if and only if $R$ is semisimple.

(2) (ii) The classes of neat (respectively coneat, $s$-pure) short exact sequences are closed under direct sums by the
proof of (1) (iii).

(iii) Since the tensor functor commutes with direct limits, it is clear that the class of $s$-pure short exact sequences
is closed under direct limits. If every simple module $S$ is finitely presented, then we have an isomorphism
$\Hom_R(S,\underset{\rightarrow}\lim \ M_i)\cong \underset{\rightarrow}\lim \ \Hom_R(S,M_i)$ for every family
$(M_i)_{i\in I}$ of modules, from which it is straightforward to show that the class of neat short exact sequences is
closed under direct limits.

(v) By \cite[Proposition~4.1]{CK1}, every right $R$-module is coneat-flat if and only if every simple right
$R$-module is injective if and only if $R$ is a right $V$-ring.
\end{proof}

\begin{cor} \cite[Proposition~3.4]{BD14}, \cite[Theorem~3.3]{Sep}, \cite[Lemma~3.1, Theorem~3.2]{BD16}
\begin{enumerate} \item The following are equivalent for a right $R$-module $X$:
\begin{enumerate}[(i)]
\item $X$ is absolutely neat.
\item $X$ is a neat submodule of an injective right $R$-module.
\item $X$ is a neat submodule of an absolutely neat right $R$-module.
\item ${\rm Ext}^1_R(S,X)=0$ for every simple right $R$-module $S$.
\end{enumerate}
\item The following are equivalent for a right $R$-module $Z$:
\begin{enumerate}[(i)]
\item $Z$ is neat-flat.
\item $Z$ is a neat quotient module of a projective right $R$-module.
\item $Z$ is a neat quotient module of a neat-flat right $R$-module.
\item For every simple right $R$-module $S$, every morphism $S\to Z$ factors through a projective right $R$-module.
\end{enumerate}
\end{enumerate} 
\end{cor}

\begin{proof} This follows by \cite[Proposition~4.6]{CK1} and \cite[Proposition~3.3]{CK1}, 
considering the exact structure given by neat short exact sequences of right $R$-modules.
\end{proof}

\begin{cor} \cite[Theorem~3.1]{BD14}
\begin{enumerate} \item The following are equivalent for a right $R$-module $X$:
\begin{enumerate}[(i)]
\item $X$ is absolutely coneat.
\item $X$ is a coneat submodule of an injective right $R$-module.
\item $X$ is a coneat submodule of an absolutely coneat right $R$-module.
\item For every simple right $R$-module $S$, every morphism $X\to S$ factors through an injective right $R$-module.
\end{enumerate}
\item The following are equivalent for a right $R$-module $Z$:
\begin{enumerate}[(i)]
\item $Z$ is coneat-flat.
\item $Z$ is a coneat quotient module of a projective right $R$-module.
\item $Z$ is a coneat quotient module of a coneat-flat right $R$-module.
\item ${\rm Ext}^1_R(Z,S)=0$ for every simple right $R$-module $S$.
\end{enumerate}
\end{enumerate} 
\end{cor}

\begin{proof} This follows by \cite[Proposition~4.6]{CK1} and \cite[Proposition~3.3]{CK1}, 
considering the exact structure given by coneat short exact sequences of right $R$-modules.
\end{proof}

\begin{cor} \cite[Lemmas 3.3, 3.4]{BD15}, \cite[Theorem~2.2]{CC}, \cite[Theorem~3.3]{Sep}
\begin{enumerate} \item The following are equivalent for a left $R$-module $X$:
\begin{enumerate}[(i)]
\item $X$ is absolutely $s$-pure.
\item $X$ is an $s$-pure submodule of an injective left $R$-module.
\item $X$ is an $s$-pure submodule of an absolutely $s$-pure left $R$-module.
\end{enumerate}
\item The following are equivalent for a left $R$-module $Z$:
\begin{enumerate}[(i)]
\item $Z$ is max-flat.
\item $Z$ is an $s$-pure quotient module of a projective left $R$-module.
\item $Z$ is an $s$-pure quotient module of a max-flat left $R$-module.
\item ${\rm Tor}_1^R(S,Z)=0$ for every simple right $R$-module $S$.
\item $Z^+$ is an absolutely neat right $R$-module.
\end{enumerate}
\end{enumerate} 
\end{cor}

\begin{proof} This follows by \cite[Proposition~4.6]{CK1} and Propositions \ref{p:Tor} and \ref{p:char}, 
considering the exact structure given by $s$-pure short exact sequences of left $R$-modules.
\end{proof}

\begin{cor} \cite[Proposition~4.9]{BD14} Every simple right $R$-module is finitely presented if and only if every
absolutely pure right $R$-module is absolutely neat. 
\end{cor}

\begin{proof} This follows by Corollary \ref{c:Mfp}, considering the exact structure given by neat short exact
sequences of right $R$-modules.
\end{proof}

\begin{cor} \cite[Theorem~4.9]{BD15} The following are equivalent:
\begin{enumerate}[(i)]
\item $R$ is a right Kasch ring.
\item Every absolutely pure right $R$-module is neat-flat.
\item Every injective right $R$-module is neat-flat.
\item The injective envelope of every simple right $R$-module is neat-flat.
\item For every free left $R$-module $Z$, $Z^+$ is a neat-flat right $R$-module.  
\end{enumerate}
\end{cor}

\begin{proof} This follows by Proposition \ref{p:Mfgp}, considering the exact structure given by neat short exact
sequences of right $R$-modules.
\end{proof}

\begin{cor} 
\begin{enumerate}
\item The following are equivalent:
\begin{enumerate}[(i)]
\item For every short exact sequence $0\to X\to Y'\to Z'\to 0$ of right $R$-modules with $X$ absolutely pure
(respectively cotorsion) and $Y'$ absolutely neat, $Z'$ is absolutely neat.
\item For every short exact sequence $0\to Z\to U\to V\to 0$ of right $R$-modules with $V$ simple and $U$ projective,
$Z$ is $FP$-projective (respectively flat).
\end{enumerate}
\item The following are equivalent:
\begin{enumerate}[(i)]
\item The class of absolutely neat right $R$-modules is closed under homomorphic images.
\item For every short exact sequence $0\to Z\to U\to V\to 0$ of right $R$-modules with $V$ simple and $U$ projective,
$Z$ is projective.
\end{enumerate}
\end{enumerate}
\end{cor}

\begin{proof} This follows by Corollaries \ref{c:fproj} and \ref{c:quot}, considering the exact structure given by neat
short exact sequences of right $R$-modules.
\end{proof}

\begin{cor} 
\begin{enumerate}
\item The following are equivalent:
\begin{enumerate}[(i)]
\item For every short exact sequence $0\to Z\to U\to V\to 0$ of right $R$-modules with $V$ $FP$-projective (respectively
flat) and $U$ coneat-flat, $Z$ is coneat-flat.
\item For every short exact sequence $0\to X\to Y'\to Z'\to 0$ of right $R$-modules with $X$ simple and $Y'$ injective,
$Z'$ is absolutely pure (respectively cotorsion).
\end{enumerate}
\item The following are equivalent:
\begin{enumerate}[(i)]
\item The class of coneat-flat right $R$-modules is closed under submodules.
\item For every short exact sequence $0\to X\to Y'\to Z'\to 0$ of right $R$-modules with $X$ simple and $Y'$ injective,
$Z'$ is injective.
\end{enumerate}
\end{enumerate}
\end{cor}

\begin{proof} This follows by Corollaries \ref{c:fproj} and \ref{c:quot}, considering the exact structure given by
coneat short exact sequences of right $R$-modules.
\end{proof}

Following Definition \ref{defcoh}, a ring $R$ will be called \emph{$\mathcal{S}$-coherent} if every simple right
$R$-module is $2$-presented.

\begin{cor} \label{c:scoh} Assume that every simple right $R$-module is finitely presented. Then the following are
equivalent: 
\begin{enumerate}[(i)] 
\item $R$ is right $\mathcal{S}$-coherent.
\item $\underset{\rightarrow}\lim {\rm Ext}^1_R(S,X_i)\cong {\rm Ext}^1_R(S,\underset{\rightarrow}\lim X_i)$ for every
simple right $R$-module $S$ and every direct system $(X_i)_{i\in I}$ of right $R$-modules.
\item The class of absolutely neat right $R$-modules is closed under direct limits.
\item The class of absolutely neat right $R$-modules is closed under pure quotients.
\item ${\rm Tor}_1^R(S,\prod_{i\in I}Z_i)\cong \prod_{i\in I} {\rm Tor}_1^R(S,Z_i)$ for every simple right $R$-module
$S$ and every family $(Z_i)_{i\in I}$ of left $R$-modules.
\item The class of max-flat left $R$-modules is closed under direct products.
\item A right $R$-module $X$ is absolutely neat if and only if $X^+$ is a max-flat left $R$-module.
\item A right $R$-module $X$ is absolutely neat if and only if $X^{++}$ is an absolutely neat right $R$-module. 
\item A left $R$-module $Z$ is max-flat if and only if $Z^{++}$ is a max-flat left $R$-module.
\end{enumerate}  
\end{cor}

\begin{proof} This follows by Proposition \ref{p:coh}, considering the exact structures given by neat short exact
sequences and $s$-pure short exact sequences of right $R$-modules.
\end{proof}

\begin{cor} Assume that every simple right $R$-module is finitely presented. Then the following are equivalent:  
\begin{enumerate}[(i)] 
\item $R$ is right $\mathcal{S}$-coherent and every max-flat left $R$-module is flat.
\item $R$ is right coherent and a right $R$-module $X$ is absolutely neat if and only if $X$ is absolutely pure.
\item A right $R$-module $X$ is absolutely neat if and only if $X^+$ is a flat left $R$-module.
\end{enumerate}  
\end{cor}

\begin{proof} This follows by Proposition \ref{p:coh2}, considering the exact structures from the proof of Corollary
\ref{c:scoh}.
\end{proof}

\section{Exact structures generated by the modules with zero Jacobson radical}

In this section we deduce a series of results concerning relatively divisible and relatively flat modules with respect
to the exact structures projectively, injectively and flatly generated by the class of modules with zero Jacobson
radical. We start with a series of examples.

\begin{exam} \rm (1) For $k\in \mathbb{N}$, denote $\mathbb{Z}_k=\mathbb{Z}/k\mathbb{Z}$. Let $p$ be a prime,
$A=\bigoplus_{n=1}^\infty \mathbb{Z}_{p^n}$ and $B=\prod_{n=1}^\infty \mathbb{Z}_{p^n}$. By \cite [Example,
p.~75]{Xu}, the short exact sequence $0\to A\to B\to B/A\to 0$ of $\mathbb{Z}$-modules is rad-pure.

(2) Since $\mathbb Z_2\otimes \mathbb Z\cong \mathbb Z_2$ and $\mathbb Z_2\otimes \mathbb Q=0$, the sequence $0\to
\mathbb Z_2\otimes \mathbb Z\to \mathbb Z_2\otimes \mathbb Q\to \mathbb Z_2 \otimes (\mathbb Q/\mathbb Z)\to 0$
of $\mathbb{Z}$-modules is not exact. So the short exact sequence $0\to \mathbb Z\to \mathbb Q\to \mathbb Q/\mathbb Z\to
0$ of $\mathbb{Z}$-modules is not rad-pure.

(3) Let $k$ be a field and $R=k[x, y]$ the polynomial ring in two indeterminates $x$ and $y$. Let $M=<x, y>$. Note that
$\Rad(M)=0$. Since $ M\otimes M\cong M^2, M\otimes R\cong M, M\otimes (R/M)=0$ and $M\ncong M^2$, the sequence $0\to
M\otimes M \to M\otimes R\to M\otimes (R/M) \to 0$ of $\mathbb{Z}$-modules is not exact. Therefore the short exact
sequence $0\to M \to R\to R/M \to 0$ of $R$-modules is not rad-pure.
\end{exam}

Clearly, every rad-pure short exact sequence is an $s$-pure short exact sequence. But the converse is not true in
general, as we may see in the following example.

\begin{exam} \label{counter} \rm Let $R$ be an integral domain having a simple module $S$ with projective dimension
$p.d.(S)>1$. Then there exists a non-splitting short exact sequence $0\to D\to M\to S\to 0$ such that $D$ is an
$h$-divisible torsion $R$-module \cite[Example~3.3]{Fuc}. As indicated in \cite[Example 3.1]{Sep}, $D$ is absolutely
coneat (hence absolutely $s$-pure, because $R$ is commutative). Since $D/\Rad (D)$ is again an
$h$-divisible torsion $R$-module, we have $T\otimes (D/\Rad (D))=0$ for every simple $R$-module $T$. Then the sequence
$$0\to (D/\Rad (D))\otimes D\to (D/\Rad (D))\otimes M\to (D/\Rad(D))\otimes S\to 0$$ is not exact. So $D$ is not an
absolutely rad-pure $R$-module.
\end{exam}

\begin{cor} \label{c:closureDER} 
\begin{enumerate} \item 
\begin{enumerate}[(i)] 
\item The classes of absolutely rad-neat (respectively absolutely rad-coneat, absolutely rad-pure) modules are closed
under extensions and rad-neat (respectively rad-coneat, rad-pure) submodules. 
\item The class of absolutely rad-neat modules is closed under direct products.
\item The classes of absolutely rad-coneat and absolutely rad-pure modules are closed under direct sums. If every
module with zero Jacobson radical is finitely generated, then the class of absolutely rad-neat modules is closed under
direct sums.
\item Every module is absolutely rad-neat (respectively absolutely rad-coneat, absolutely rad-pure) if and only if
every short exact sequence of modules is rad-neat (respectively rad-coneat, rad-pure). 
\item Every module is absolutely rad-neat if and only if every module with zero Jacobson radical is projective. 
\end{enumerate}
\item 
\begin{enumerate}[(i)]
\item The classes of rad-neat-flat (respectively rad-coneat-flat, rad-pure-flat) modules are closed under extensions
and rad-neat (respectively rad-coneat, rad-pure) quotients. 
\item The classes of rad-coneat-flat and rad-pure-flat modules are closed under direct sums. If every module with zero
Jacobson radical is finitely generated, then the class of rad-neat flat modules is closed under direct sums.
\item The class of rad-pure-flat modules is closed under direct limits. If every module with zero Jacobson radical is
finitely presented, then the class of rad-neat-flat modules is closed under direct limits.
\item Every module is rad-neat-flat (respectively rad-coneat-flat, rad-pure-flat) if and only if every short exact
sequence of modules is rad-neat (respectively rad-coneat, rad-pure).
\item Every module is rad-coneat-flat if and only if every module with zero Jacobson radical is injective.
\end{enumerate}
\end{enumerate}
\end{cor}

\begin{proof} This mainly follows by \cite[Proposition~4.1]{CK1}, considering the exact structures given by rad-neat
(respectively rad-coneat, rad-pure) short exact sequences of modules, which are projectively (respectively
injectively, flatly) generated by the class of modules with zero Jacobson radical. We only add some details.

(1) (ii) Using the isomorphism $\Hom_R(M,\prod_{i\in I}M_i)\cong \prod_{i\in I}\Hom_R(M,M_i)$ for every module $M$ and
every family $(M_i)_{i\in I}$ of modules, it is straightforward to show that the class of rad-neat short exact sequences
is closed under direct products. 

(iii) One can easily show that the classes of rad-coneat and rad-pure short exact sequences are closed under direct sums
by \cite[16.2]{Wisb} and \cite[12.15]{Wisb} respectively. If every module $M$ with zero Jacobson radical is finitely
generated, then we have an isomorphism $\Hom_R(M,\bigoplus_{i\in I}M_i)\cong \bigoplus_{i\in I}\Hom_R(M,M_i)$ for every
family $(M_i)_{i\in I}$ of modules, from which it is straightforward (or use \cite[18.2]{Wisb}) to show that the class
of rad-neat short exact sequences is closed under direct sums. 

(2) (ii) By the proof of (1) (iii), the classes of rad-coneat and rad-pure short exact sequences are always closed under
direct sums, while the class of rad-neat short exact sequences is closed under direct sums if every module with zero
Jacobson radical is finitely generated. 

(iii) Since the tensor functor commutes with direct limits, it is clear that the class of rad-pure short exact sequences
is closed under direct limits. If every module $M$ with zero Jacobson radical is finitely presented, then we have an
isomorphism $\Hom_R(M,\underset{\rightarrow}\lim \ M_i)\cong \underset{\rightarrow}\lim \ \Hom_R(M,M_i)$ for every
family $(M_i)_{i\in I}$ of modules, from which it is straightforward to show that the class of rad-neat short exact
sequences is closed under direct limits.
\end{proof}

\begin{cor} 
\begin{enumerate} \item The following are equivalent for a right $R$-module $X$:
\begin{enumerate}[(i)]
\item $X$ is absolutely rad-neat.
\item $X$ is a rad-neat submodule of an injective right $R$-module.
\item $X$ is a rad-neat submodule of an absolutely rad-neat right $R$-module.
\item ${\rm Ext}^1_R(M,X)=0$ for every right $R$-module $M$ with zero Jacobson radical.
\end{enumerate}
\item The following are equivalent for a right $R$-module $Z$:
\begin{enumerate}[(i)]
\item $Z$ is rad-neat-flat.
\item $Z$ is a rad-neat quotient module of a projective right $R$-module.
\item $Z$ is a rad-neat quotient module of a rad-neat-flat right $R$-module.
\item For every right $R$-module $M$ with zero Jacobson radical, every morphism $M\to Z$ factors through a projective
right $R$-module.
\end{enumerate}
\end{enumerate} 
\end{cor}

\begin{proof} This follows by \cite[Proposition~4.6]{CK1} and \cite[Proposition~3.3]{CK1}, considering the exact structure given by
rad-neat short exact sequences of right $R$-modules.
\end{proof}

\begin{cor} \label{applic}
\begin{enumerate} \item The following are equivalent for a right $R$-module $X$:
\begin{enumerate}[(i)]
\item $X$ is absolutely rad-coneat.
\item $X$ is a rad-coneat submodule of an injective right $R$-module.
\item $X$ is a rad-coneat submodule of an absolutely rad-coneat right $R$-module.
\item For every right $R$-module $M$ with zero Jacobson radical, every morphism $X\to M$ factors through an injective
right $R$-module.
\end{enumerate}
\item The following are equivalent for a right $R$-module $Z$:
\begin{enumerate}[(i)]
\item $Z$ is rad-coneat-flat.
\item $Z$ is a rad-coneat quotient module of a projective right $R$-module.
\item $Z$ is a rad-coneat quotient module of a rad-coneat-flat right $R$-module.
\item ${\rm Ext}^1_R(Z,M)=0$ for every right $R$-module $M$ with zero Jacobson radical.
\end{enumerate}
\end{enumerate} 
\end{cor}

\begin{proof} This follows by \cite[Proposition~4.6]{CK1} and \cite[Proposition~3.3]{CK1}, considering the exact structure given by
rad-coneat short exact sequences of right $R$-modules.
\end{proof}

\begin{cor}
\begin{enumerate} \item The following are equivalent for a left $R$-module $X$:
\begin{enumerate}[(i)]
\item $X$ is absolutely rad-pure.
\item $X$ is a rad-pure submodule of an injective left $R$-module.
\item $X$ is a rad-pure submodule of an absolutely rad-pure left $R$-module.
\end{enumerate}
\item The following are equivalent for a left $R$-module $Z$:
\begin{enumerate}[(i)]
\item $Z$ is rad-pure-flat.
\item $Z$ is a rad-pure quotient module of a projective left $R$-module.
\item $Z$ is a rad-pure quotient module of a rad-pure-flat left $R$-module.
\item ${\rm Tor}_1^R(M,Z)=0$ for every right $R$-module $M$ with zero Jacobson radical.
\item $Z^+$ is an absolutely rad-coneat right $R$-module.
\end{enumerate}
\end{enumerate}
\end{cor}

\begin{proof} This follows by \cite[Proposition~4.6]{CK1} and Propositions \ref{p:Tor} and \ref{p:char}, 
considering the exact structure given by rad-pure short exact sequences of left $R$-modules.
\end{proof}

\begin{cor} 
\begin{enumerate}
\item The following are equivalent:
\begin{enumerate}[(i)]
\item For every short exact sequence $0\to X\to Y'\to Z'\to 0$ of right $R$-modules with $X$ absolutely pure
(respectively cotorsion) and $Y'$ absolutely rad-neat, $Z'$ is absolutely rad-neat.
\item For every short exact sequence $0\to Z\to U\to V\to 0$ of right $R$-modules with $V$ having zero Jacobson radical
and $U$ projective, $Z$ is $FP$-projective (respectively flat).
\end{enumerate}
\item The following are equivalent:
\begin{enumerate}[(i)]
\item The class of absolutely rad-neat right $R$-modules is closed under homomorphic images.
\item For every short exact sequence $0\to Z\to U\to V\to 0$ of right $R$-modules with $V$ having zero Jacobson
radical and $U$ projective, $Z$ is projective.
\end{enumerate}
\end{enumerate}
\end{cor}

\begin{proof} This follows by Corollaries \ref{c:fproj} and \ref{c:quot}, considering the exact structure given by
rad-neat short exact sequences of right $R$-modules.
\end{proof}

\begin{cor} 
\begin{enumerate}
\item The following are equivalent:
\begin{enumerate}[(i)]
\item For every short exact sequence $0\to Z\to U\to V\to 0$ of right $R$-modules with $V$ $FP$-projective (respectively
flat) and $U$ rad-coneat-flat, $Z$ is rad-coneat-flat.
\item For every short exact sequence $0\to X\to Y'\to Z'\to 0$ of right $R$-modules with $X$ having zero Jacobson
radical and $Y'$ injective, $Z'$ is absolutely pure (respectively cotorsion).
\end{enumerate}
\item The following are equivalent:
\begin{enumerate}[(i)]
\item The class of rad-coneat-flat right $R$-modules is closed under submodules.
\item For every short exact sequence $0\to X\to Y'\to Z'\to 0$ of right $R$-modules with $X$ having zero Jacobson
radical and $Y'$ injective, $Z'$ is injective.
\end{enumerate}
\end{enumerate}
\end{cor}

\begin{proof} This follows by Corollaries \ref{c:fproj} and \ref{c:quot}, considering the exact structure given by
rad-coneat short exact sequences of right $R$-modules.
\end{proof}

Following Definition \ref{defcoh}, a ring $R$ will be called \emph{rad-coherent} if every right $R$-module with zero
Jacobson radical is $2$-presented.

\begin{cor} \label{c:radcoh} Assume that every right $R$-module with zero Jacobson radical is finitely presented. Then
the following are equivalent: 
\begin{enumerate}[(i)] 
\item $R$ is right rad-coherent.
\item $\underset{\rightarrow}\lim {\rm Ext}^1_R(M,X_i)\cong {\rm Ext}^1_R(M,\underset{\rightarrow}\lim X_i)$ for every
right $R$-module $M$ with zero Jacobson radical and every direct system $(X_i)_{i\in I}$ of right $R$-modules.
\item The class of absolutely rad-neat right $R$-modules is closed under direct limits.
\item The class of absolutely rad-neat right $R$-modules is closed under pure quotients.
\item ${\rm Tor}_1^R(M,\prod_{i\in I}Z_i)\cong \prod_{i\in I} {\rm Tor}_1^R(M,Z_i)$ for every right $R$-module $M$
with zero Jacobson radical and every family $(Z_i)_{i\in I}$ of left $R$-modules.
\item The class of rad-pure-flat left $R$-modules is closed under direct products.
\item A right $R$-module $X$ is absolutely rad-neat if and only if $X^+$ is a rad-pure-flat left $R$-module.
\item A right $R$-module $X$ is absolutely rad-neat if and only if $X^{++}$ is an absolutely rad-neat right
$R$-module. 
\item A left $R$-module $Z$ is rad-pure-flat if and only if $Z^{++}$ is a rad-pure-flat left $R$-module.
\end{enumerate}  
\end{cor}

\begin{proof} This follows by Proposition \ref{p:coh}, considering the exact structures given by rad-neat 
and rad-pure short exact sequences of right $R$-modules.
\end{proof}

\begin{cor} Assume that every right $R$-module with zero Jacobson radical is finitely presented. Then the following are
equivalent:  
\begin{enumerate}[(i)] 
\item $R$ is right rad-coherent and every rad-pure-flat left $R$-module is flat.
\item $R$ is right coherent and a right $R$-module $X$ is absolutely rad-neat if and only if $X$ is absolutely pure.
\item A right $R$-module $X$ is absolutely rad-neat if and only if $X^+$ is a flat left $R$-module.
\end{enumerate}  
\end{cor}

\begin{proof} This follows by Proposition \ref{p:coh2}, considering the exact structures from the proof of Corollary
\ref{c:radcoh}.
\end{proof}

\end{document}